\documentclass[12pt]{amsart}
\usepackage[alphabetic]{amsrefs}
\usepackage{mathdots, blkarray}
\usepackage{multicol}
\usepackage{graphicx}
\usepackage{caption}
\usepackage{amscd}
\usepackage{upgreek}
\usepackage{stmaryrd}
\usepackage{longtable}
\usepackage[T1]{fontenc}
\usepackage{latexsym, amsmath, amssymb, amsthm}
\usepackage[svgnames]{xcolor}
\usepackage{rsfso}
\usepackage[mathscr]{eucal}
\usepackage{mathtools}
\usepackage{mathptmx}
\usepackage{titletoc}
\usepackage{wrapfig}
\usepackage{float}
\usepackage{xypic}
\usepackage{microtype}
\usepackage{dsfont}
\usepackage{xcolor}
\usepackage{color}
\usepackage[colorlinks = true,
            linkcolor  = DarkBlue,
            urlcolor   = DarkRed,
            citecolor  = DarkGreen]{hyperref}
\usepackage{hyperref}
\allowdisplaybreaks	
\usepackage[centering, includeheadfoot, hmargin=1.0in, tmargin=0.5in, 
  bmargin=0.6in, headheight=6pt]{geometry}
\newtheorem{theorem}{Theorem}[section]
\newtheorem{prop}[theorem]{Proposition}

\newtheorem{lem}[theorem]{Lemma}
\newtheorem{coro}[theorem]{Corollary}

\newtheorem{thm}[theorem]{Theorem}

\newtheorem{rem}[theorem]{\rm\textsc{Remark}}
\newtheorem{exam}[theorem]{\rm\textsc{Example}}

\newcommand{\ideal}[1]{\ensuremath{\left\langle #1 \right\rangle}}

\newcommand{\bslash}{\kern-0.1em\texttt{\scalebox{0.6}[1]{/}}\kern-0.15em \texttt{\scalebox{0.6}[1]{/}}}

\DeclareMathOperator{\cent}{Cent}

\DeclareMathOperator{\gder}{GDer}
\DeclareMathOperator{\qder}{QDer}
\DeclareMathOperator{\zder}{ZDer}
\DeclareMathOperator{\der}{Der}
\DeclareMathOperator{\qcent}{QCent}
\newcommand{\A}{\mathcal{A}} 
 
\newcommand{\B}{\mathcal{B}} 
\newcommand{\C}{\mathbb{C}}

\newcommand{\g}{\mathfrak{g}} 
\newcommand{\gl}{\mathfrak{gl}}

\newcommand{\ra}{\longrightarrow}

\newcommand{\wt}{\widetilde}

\newcommand{\hbo}{$\hfill\Diamond$} 

\begin{document}
\title{Generalized derivations of $\upomega$-Lie algebras} 
\def\shorttitle{Generalized derivations of $\upomega$-Lie algebras}

\author{Yin Chen}
\address{School of Mathematics and Physics, Jinggangshan University,
Ji'an 343009, Jiangxi, China \& Department of Finance and Management Science, University of Saskatchewan, Saskatoon, SK, Canada, S7N 5A7}
\email{yin.chen@usask.ca}

\author{Shan Ren}
\address{School of Mathematics and Statistics, Northeast Normal University, Changchun 130024, China}
\email{rens734@nenu.edu.cn}

\author{Jiawen Shan}
\address{School of Mathematics and Systems Science, Shenyang Normal University, Shenyang 110034, China}
\email{ShanJiawen826@outlook.com}

\author{Runxuan Zhang}
\address{Department of Mathematics and Information Technology, Concordia University of Edmonton, Edmonton, AB, Canada, T5B 4E4}
\email{runxuan.zhang@concordia.ab.ca}

\begin{abstract}
This article explores the structure theory of compatible generalized derivations of finite-dimensional $\upomega$-Lie algebras over a field $\mathbb{K}$. We prove that any compatible quasiderivation of an $\upomega$-Lie algebra
can be embedded  as a compatible derivation into a larger $\upomega$-Lie algebra, refining the general result established by Leger and Luks in 2000 for finite-dimensional nonassociative algebras. We also provide an approach to explicitly compute (compatible) generalized derivations and quasiderivations for all $3$-dimensional non-Lie complex $\upomega$-Lie algebras.
\end{abstract}

\date{\today}
\maketitle \baselineskip=16.2pt

\dottedcontents{section}[1.16cm]{}{1.8em}{5pt}
\dottedcontents{subsection}[2.00cm]{}{2.7em}{5pt}

\section{Introduction}
\setcounter{equation}{0}
\renewcommand{\theequation}
{1.\arabic{equation}}
\setcounter{theorem}{0}
\renewcommand{\thetheorem}
{1.\arabic{theorem}}

\noindent The study of finite-dimensional nonassociative algebras over a field $\mathbb{K}$ attached with a bilinear form $\upomega$ is a new trend in many topics related to Lie theory; see \cite{CNY23b, CW23, ZCMS18} and \cite{ZC21}. As a natural generalization of Lie algebras, the algebraic theory of $\upomega$-Lie algebras was initiated by \cite{Nur07} and originally aims to understand the geometry of isoparametric hypersufaces in Riemannian geometry. In the past 10 years, algebraic structures of finite-dimensional complex  $\upomega$-Lie algebras have been studied extensively and many classical results in Lie algebras were generalized to the cases of $\upomega$-Lie algebras; see for example, \cite[Corollary 1.3]{Zha21} for an $\upomega$-version of Lie's theorem.  Among these  studies, derivation theory plays an essential and irreplaceable role in understanding structures, representations, extensions, and automorphism groups of $\upomega$-Lie algebras; see \cite{CZZZ18} and \cite{Oub24}.

A systematic generalization on classical derivation theory of Lie algebras  dates back to \cite{LL00} which introduced the concept of generalized derivations for any finite-dimensional nonassociative algebras and thoroughly studied the tower formed by derivation algebras, generalized derivations, and the entire space of all linear maps. This tower framework and the philosophy of exploring the coincidence of two terms have inspired numerous subsequent research; see  \cite{CCZ21, CRSZ24} and \cite{CZ23}. The primary objective of this article is to develop a theory of generalized derivations for finite-dimensional $\upomega$-Lie algebras, approaching it from both the algebraic structure and computational aspects.

Let us recall some fundamental concepts and facts on $\upomega$-Lie algebras. Let $\mathbb{K}$ be a field of characteristic zero, $L$ a finite-dimensional vector space over $\mathbb{K}$, and 
$\upomega$ be a skew-symmetric bilinear form on $L$. We say that $L$ is an \textit{$\upomega$-Lie algebra} over $\mathbb{K}$ if there exists a skew-symmetric bilinear bracket $[-,-]:L\times L\ra L$ such that the following \textit{$\upomega$-Jacobi identity} holds:
\begin{equation}\tag{$\upomega$-Jacobi identity}\label{Jacobi}
[[x,y],z]+[[y,z],x]+[[z,x],y]=\upomega(x,y)\cdot z+\upomega(y,z)\cdot x+\upomega(z,x)\cdot y
\end{equation}
for all $x,y,z\in L$. Apparently, the class of Lie algebras over $\mathbb{K}$ coincides with the class of 
$\upomega$-Lie algebras with $\upomega=0$. All non-Lie complex $\upomega$-Lie algebras of dimension at most $5$ have been classified by \cite{CLZ14, CZ17} and \cite{CNY23a}. These low-dimensional examples are  indispensable because they usually provide valuable insights into the behavior of the class of $\upomega$-Lie algebras. For instance, finite-dimensional non-Lie simple complex $\upomega$-Lie algebras must be $3$-dimensional; see \cite[Theorem 7.1]{CZ17}. 

A linear map $f$ on a finite-dimensional $\upomega$-Lie algebra $L$ is called a \textit{generalized derivation} of $L$ provided that there exist two linear maps $f_1,f_2$ on $L$ such that
\begin{equation}
\label{gder}
[f(x),y]=f_2([x,y])-[x,f_1(y)]
\end{equation}
for all $x,y\in L$. Clearly, a (classical) derivation $f$ of $L$ is a special case of generalized derivations for which $f_1=f_2=f$ in (\ref{gder}). We write $\gder(L)$ for the set of all generalized derivations of $L$ and it is nonempty because it contains the derivation algebra $\der(L)$ of $L$. In particular,  $\gder(L)$ is also a Lie subalgebra of $\gl(L)$; see \cite[Lemma 3.1 (i)]{LL00}. A linear map $f$ on $L$ is \textit{compatible} if 
\begin{equation} \label{comp}
\upomega(f(x),y)+\upomega(x,f(y))=0.
\end{equation}
for all $x,y\in L$. In \cite[Proposition 2.4]{CZZZ18}, we have already seen that the subset $\der_c(L)$ of $\der(L)$ consisting of all compatible derivations of $L$ is a vector space. Moreover, there exists a close connection between $\der_c(L)$ and the automorphism group of $L$; see \cite[Sections 3 and 4]{CZZZ18}.

This naturally motivates our interest in understanding the structure of compatible generalized derivations of $L$.
Write $\gder_c(L)$ for the set of all compatible generalized derivations on $L$. We will see below that $\gder_c(L)$ is a vector space. Then we have the following tower of vector spaces:
\begin{equation}\label{tower}
\der_c(L)\subseteq \gder_c(L)\subseteq \gder(L)\subseteq \gl(L).
\end{equation}  
As the first  topic of this article, exploring the difference between two terms in the tower (\ref{tower}) above is entirely reasonable, because it clearly inherits the philosophy occurred in \cite{LL00}. 
To better understand such differences, we also need to consider the space $\qder_c(L)$ of compatible quasiderivations of $L$; see Section \ref{sec2} below. Note that methods and viewpoints in  \cite{LL00} have influenced many recent research; for example, \cite{Vla22, BPS24}, and \cite{Vla24}.

A classical result in the theory of generalized derivations of Lie algebras  states that 
any quasiderivation of a Lie algebra can be embedded as a derivation into a larger Lie algebra; see \cite[Section 3]{LL00}. This result have been extended to many cases of generalizations of Lie algebras such as Lie superalgebras \cite{ZZ10} and Hom-Lie triple systems \cite{ZCM18}. Our second topic is to generalize and to refine this classical result on compatible quasiderivations to the case of $\upomega$-Lie algebras and to explore the difference between the image of all compatible quasiderivations of an $\upomega$-Lie algebra and the compatible derivation algebra of the larger $\upomega$-Lie algebra.  

The third topic of this article aims to develop an approach to explicitly calculate compatible generalized derivations of low-dimensional complex $\upomega$-Lie algebras. Literally speaking, this approach comes from commutative algebra. In fact, many methods from commutative algebras and computational algebraic geometry have been applied recently to the study of many nonassociative algebras; see for example, \cite{Che25, CZ24a,CZ24b} and \cite{RZ24}. Note that our computations will be based on the classifications appeared in \cite{CLZ14} and \cite{CZ17}.

We organize this article as follows.  Section \ref{sec2} contains a fundamental exploration on generalized derivations of $\upomega$-Lie algebras. We prove two key lemmas on compatible linear maps and associated maps $f'$ to a quasiderivation $f$; see Lemmas \ref{keyl} and \ref{lem2}, respectively. We also obtain a result on decompositions of compatible generalized derivations (Proposition \ref{prop2.5}) and present a specific $3$-dimensional example to explicitly illustrate the difference between these generalized derivations; see Example \ref{exam2.6}.
In Section \ref{sec3}, we refine the result in \cite[Proposition 3.5]{LL00}, demonstrating that
every compatible quasiderivation of an $\upomega$-Lie algebra $L$ can be embedded as a compatible derivation into a larger $\upomega$-Lie algebra $\wt{L}$. Moreover, we derive a decomposition theorem on $\der_c(\wt{L})$; see Theorem \ref{mainthm}. Section \ref{sec4} is devoted to explicit computations on 
generalized derivations, compatible generalized derivations, quasiderivations, and compatible quasiderivations of all $3$-dimensional non-Lie complex $\upomega$-Lie algebras.

Throughout this article, $\mathbb{K}$ denotes a field of characteristic zero, $\C$ denotes the field of complex field,  and all algebras are finite-dimensional over $\mathbb{K}$ or $\C$.

\section{Generalized Derivations} \label{sec2}
\setcounter{equation}{0}
\renewcommand{\theequation}
{2.\arabic{equation}}
\setcounter{theorem}{0}
\renewcommand{\thetheorem}
{2.\arabic{theorem}}

\noindent  This section contains some fundamental notions and facts on generalized derivations of  $\upomega$-Lie algebras. Some statements might have appeared in \cite[Section 3]{LL00} but there were no detailed proofs provided.  Here we will modify and refine these statements for $\upomega$-Lie algebras and present detailed proofs.

Let $\mathbb{K}$ be a field of characteristic zero and  $(L,[-,-],\upomega)$ be a finite-dimensional $\upomega$-Lie algebra over $\mathbb{K}$. We write $\gl(L)$ for the general linear Lie algebra on $L$.

\begin{lem}\label{keyl}
Let $a\in \mathbb{K}$ be a scalar and $f,g\in \gl(L)$ be any two compatible linear maps. Then $f+g, a\cdot f,$ and $[f,g]$  are also compatible.
\end{lem}

\begin{proof}
Since $\upomega$ is bilinear, it is easy to see that both $f+g$ and $a\cdot f$ are compatible. To see that $[f,g]$ is compatible, note that $[f,g]=fg-gf$ in $\gl(L)$, thus for all $x,y\in L$, we have
\begin{eqnarray*}
\upomega([f,g](x),y) & = & \upomega(fg(x)-gf(x),y) = \upomega(fg(x),y)-\upomega(gf(x),y)\\
 &=&-\upomega(g(x),f(y))+\upomega(f(x),g(y)).
\end{eqnarray*}
Similarly, $\upomega(x,[f,g](y))=-\upomega(f(x),g(y))+\upomega(g(x),f(y)).$ Hence,
$$\upomega([f,g](x),y)+\upomega(x,[f,g](y))=0.$$
Namely, $[f,g]$ is also compatible.
\end{proof}

Recall that a generalized derivation $f$ of $L$ is called a \textit{quasiderivation} of $L$ if $f_1=f$ in (\ref{gder}). In other words,  there exists a linear map $f'$ on $L$ such that
\begin{equation}
\label{qder}
[f(x),y]+[x,f(y)]=f'([x,y])
\end{equation}
for all $x,y\in L$. 
By \cite[Lemma 3.1]{LL00}, we see that the set $\qder(L)$ of all quasiderivations of $L$ is nonempty and it is actually a Lie subalgebra of $\gder(L)$. We write $\qder_c(L)$ for the set of all compatible quasiderivations of $L$ and clearly, we have
\begin{equation}\label{tower2}
\der_c(L)\subseteq \qder_c(L)\subseteq \gder_c(L)\subseteq \gder(L).
\end{equation}  

We denote by $\qcent_c(L)$ the \textit{compatible quasicentroid} of $L$, which is defined as the set of consisting of all compatible linear maps $f$ on $L$ such that 
\begin{equation} \label{qcent}
[f(x),y]=[x,f(y)]
\end{equation}
for all $x,y\in L$. Furthermore, we write $\cent_c(L)$ for the  \textit{compatible centroid} of $L$, defined as the set consisting of all compatible linear maps $f$ on $L$ satisfying
\begin{equation}
\label{cent}
[f(x),y]=[x,f(y)]=f([x,y])
\end{equation}
for all $x,y\in L$. 

\begin{prop}\label{prop2.2}
The set $\gder_c(L)$ is a Lie subalgebra of $\gder(L)$ and $\qder_c(L)$
is also a Lie subalgebra of $\gder_c(L)$. 
\end{prop}

\begin{proof}
Clearly, $\gder_c(L)$ is a nonempty subset of $\gder(L)$ as it contains the zero map.  Since $\gder(L)$ is a Lie subalgebra of $\gl(L)$, it follows that $f+g,a\cdot f$, and $[f,g]\in \gder(L)$
for all  $a\in \mathbb{K}$ and all $f,g\in \gder_c(L)$. By Lemma \ref{keyl}, 
$f+g,a\cdot f$, and $[f,g]$ are compatible generalized derivations of $L$. Hence, $\gder_c(L)$ is a Lie subalgebra of $\gder(L)$. 

The second statement also can be obtained by combining Lemma \ref{keyl} and the fact in \cite[Lemma 3.1 (i)]{LL00}, which states that  $\qder(L)$ is a Lie subalgebra of $\gder(L)$.
\end{proof}

In \cite{LL00}, the authors also use the pair $(f,f')$ to denote a quasiderivation of a nonassociative algebra, where $f'$ is a linear map associated with $f$ in (\ref{qder}).  We will see in Proposition \ref{prop3.3} below that
determining such $f'$ plays a key role in studying the embedding question of quasiderivations.

\begin{lem}\label{lem2}
Let $c\in k$ and $f,g\in\qder(L)$. Then 
\begin{enumerate}
  \item $(c\cdot f)'=c\cdot f'$ and $(f+g)'=f'+g'$.
  \item $[f,g]'=[f',g']$.
\end{enumerate}
\end{lem}

\begin{proof} Suppose that $x,y\in L$ denote two arbitrary elements. 

(1) As $[f(x),y]+[x,f(y)]=f'([x,y])$, it follows that $[c\cdot f(x),y]+[x,c\cdot f(y)]=c([f(x),y]+[x,f(y)])=c(f'([x,y]))=
(c\cdot f')[x,y]$. Thus $(c\cdot f)'=c\cdot f'$. Similarly,
\begin{eqnarray*}
[(f+g)(x),y]+[x,(f+g)(y)] & = & [f(x),y]+[g(x),y]+[x,f(y)]+[x,g(y)] \\
 &=& [f(x),y]+[x,f(y)]+[g(x),y]+[x,g(y)] \\
 & = & f'([x,y])+g'([x,y])=(f'+g')([x,y]).
\end{eqnarray*}
This means that $(f+g)'=f'+g'$.

(2) To prove $[f,g]'=[f',g']$, we note that $f,g\in\qder(L)\subseteq\gl(L)$ and $[f,g]=fg-gf$. Moreover, 
\begin{eqnarray*}
[f,g]'([x,y])& = & [[f,g](x),y] + [x,[f,g](y)] \\
 & = & [fg(x),y]-[gf(x),y]+[x,fg(y)]-[x,gf(y)]\\
 &=& [fg(x),y]+[g(x),f(y)]-[g(x),f(y)]-[x,gf(y)]+\\
 &&[x,fg(y)]+[f(x),g(y)]-[f(x),g(y)]-[gf(x),y]\\
 &=&f'([g(x),y])-g'([x,f(y)])+f'([x,g(y)])-g'([f(x),y])\\
 &=&f'\Big([g(x),y]+[x,g(y)]\Big)-g'\Big([x,f(y)]+[f(x),y]\Big)\\
 &=& f'(g'([x,y]))-g'(f'([x,y]))=[f',g']([x,y]).
\end{eqnarray*}
Therefore, $[f,g]'=[f',g']$.
\end{proof}
 
\begin{rem}{\rm
The following statements hold immediately from the last four statements in \cite[Lemma 3.1]{LL00} and Lemma \ref{keyl} above.
\begin{enumerate}
  \item $[\der_c(L),\cent_c(L)]\subseteq \cent_c(L)$.
  \item $[\qder_c(L),\qcent_c(L)]\subseteq\qcent_c(L)$.
  \item $\cent_c(L)\subseteq\qcent_c(L)$.
  \item $[\qcent_c(L),\qcent_c(L)]\subseteq \qder_c(L)$.
\end{enumerate}
One may obtain detailed proofs by applying a similar argument in Proposition \ref{prop2.2}.
\hbo}\end{rem}

\begin{prop}\label{prop2.5}
If $\qder_c(L)=\qder(L)$ or $\qcent_c(L)=\qcent(L)$, then
$$\gder_c(L)=\qder_c(L)+\qcent_c(L).$$
\end{prop}

\begin{proof}
We prove this result by showing the mutual inclusion of both sides.
To prove $\qder_c(L)+\qcent_c(L)\subseteq \gder_c(L)$, we recall \cite[Proposition 3.3 (1)]{LL00}, which asserts that $\qder(L)+\qcent(L)=\gder(L)$. Thus $\qder_c(L)+\qcent_c(L)\subseteq \gder(L)$. Now, applying 
Lemma \ref{keyl} obtains $\qder_c(L)+\qcent_c(L)\subseteq \gder_c(L)$. 

To prove the inverse inclusion, we take an arbitrary $f\in \gder_c(L)$. Then there exist two linear maps $f_1,f_2$ on $L$ such that $[f(x),y]=f_2([x,y])-[x,f_1(y)]$ for all $x,y\in L$. By the anti-commutativity of $\upomega$-Lie algebras, we see that $[f_1(y),x]=-f_2([x,y])+[f(x),y]=f_2([y,x])-[y,f(x)]$. Thus, $f_1$ is also a generalized derivation. We define
$$u:=\frac{f+f_1}{2}\textrm{ and }v:=\frac{f-f_1}{2}.$$
Clearly, $f=u+v$. Furthermore, $u\in \qder(L)$ because
\begin{eqnarray*}
[u(x),y]+[x,u(v)] & = &\left[\frac{f(x)+f_1(x)}{2},y\right]+ \left[x,\frac{f(y)+f_1(y)}{2}\right] \\
 & = & \frac{1}{2}\left([f(x),y]+[x,f_1(y)]\right)+\frac{1}{2}\left([x,f(y)]+[f_1(x),y]\right)\\
 &=&\frac{1}{2}\cdot f_2([x,y])+\frac{1}{2}\cdot f_2([x,y])=f_2([x,y]).
\end{eqnarray*}
To see that $v\in \qcent(L)$, we note that 
$$[v(x),y]-[x,v(y)]=\left[\frac{f(x)-f_1(x)}{2},y\right]-\left[x,\frac{f(y)-f_1(y)}{2}\right]=
\frac{1}{2}\cdot f_2([x,y])-\frac{1}{2}\cdot f_2([x,y])=0.$$
As $\qder_c(L)=\qder(L)$ or $\qcent_c(L)=\qcent(L)$, 
it follows that either $u$ is compatible or $v$ is compatible. By Lemma \ref{keyl}, we see that $u$ and $v$ both are compatible. Hence, $f\in \qder_c(L)+\qcent_c(L)$, as desired.
\end{proof}

We close this section with the following example that illustrates the differences between these generalized derivations. 

\begin{exam}\label{exam2.6}
{\rm
Recall that the $3$-dimensional non-Lie complex $\upomega$-Lie algebra $L_1$ in \cite[Theorem 2]{CLZ14}, which can be spanned by $\{x,y,z\}$ subject to the following  generating relations:
$$[x,y]=y,[x,z]=0,[y,z]=z\textrm{ and }\upomega(x,y)=1, \upomega(x,z)=\upomega(y,z)=0.$$

(1) Let's first compute the space $\gder(L_1)$. Suppose that $$f=\begin{pmatrix}
     x_{11} & x_{12} &x_{13}   \\
     x_{21} & x_{22} &x_{23}   \\
     x_{31} & x_{32} &x_{33}   \\
\end{pmatrix}, f_1=\begin{pmatrix}
     a_{11} & a_{12} &a_{13}   \\
     a_{21} & a_{22} &a_{23}   \\
     a_{31} & a_{32} &a_{33}   \\
\end{pmatrix},\textrm{and }f_2=\begin{pmatrix}
     b_{11} & b_{12} &b_{13}   \\
     b_{21} & b_{22} &b_{23}   \\
     b_{31} & b_{32} &b_{33}   \\
\end{pmatrix}$$
where the action of $f$ on $L_1$ is given by
\begin{eqnarray*}
f(x)&=&x_{11}\cdot x+x_{21}\cdot y+x_{31}\cdot z\\
f(y)&=&x_{12}\cdot x+x_{22}\cdot y+x_{32}\cdot z\\
f(z)&=&x_{13}\cdot x+x_{23}\cdot y+x_{33}\cdot z.
\end{eqnarray*}
The actions of $f_1$ and $f_2$ on $L_1$ are defined in the similar way. Substituting these actions into
(\ref{gder}) and verifying  (\ref{gder}) for $(x,y),(y,x), (x,z),(z,x)$ and $(y,z),(z,y)$ respectively, we obtain
\begin{equation} \label{eq2.5}
\begin{aligned}
&x_{21}=x_{23}=0\\
&x_{11}+a_{22}-b_{22}=x_{13} + b_{23}=x_{22} + a_{33} - b_{33}=x_{31} +b_{32}=x_{33} + a_{22} - b_{33}=    0
\\
&a_{11} - a_{33} - b_{22} + b_{33}=a_{13} + b_{23}=a_{21}=a_{23}= a_{31} + b_{32}=b_{12}=b_{13}=0.
\end{aligned}
\end{equation}

Note that $x_{11}, x_{13}, x_{22},x_{31}$ and $x_{33}$ can be expressed by the set $\{a_{ij},b_{ij}\mid 1\leqslant i,j\leqslant 3\}$ of free variables, thus they also can be viewed as free variables. There are two equations $x_{21}=0=x_{23}$
only involving $x_{ij}$, so they indicate that
$$\dim(\gder(L_1))=7.$$
More precisely, a generic map $$f=\begin{pmatrix}
     x_{11} & x_{12} &x_{13}   \\
     0 & x_{22} &0   \\
     x_{31} & x_{32} &x_{33}   \\
\end{pmatrix},$$
together with
$$f_1=\begin{pmatrix}
     a_{33} + b_{22} - b_{33} & a_{12} &-b_{23}   \\
     0& b_{22}-x_{11} &0   \\
     -b_{32} & a_{32} &a_{33}   \\
\end{pmatrix}\textrm{and }f_2=\begin{pmatrix}
     b_{11} & 0 &0  \\
     b_{21} & b_{22} &b_{23}   \\
     b_{31} & b_{32} &b_{33}   \\
\end{pmatrix}$$
form a generalized derivation of $L_1$. Therefore,
$$\gder(L_1)\neq \gl(L_1)$$
because $\dim(\gl(L_1))=9.$

(2) Secondly, we may compute $\gder_c(L_1)$ by combining (\ref{comp}) and (\ref{gder}) for the pairs of basis elements: $(x,y), (x,z),$ and $(y,z)$. In fact, applying (\ref{comp}) for $(x,y), (x,z),$ and $(y,z)$ obtains
$$x_{22}+x_{11}=x_{23}=x_{21}=x_{13}=0.$$
Hence, an element $f\in \gder_c(L_1)$ must be of  the following form
$$f=\begin{pmatrix}
     x_{11} & x_{12} &0   \\
     0 & -x_{11} &0   \\
     x_{31} & x_{32} &x_{33}   \\
\end{pmatrix}.$$
This also shows that $\gder_c(L_1)$ is a $5$-dimensional vector space, and thus
$$\gder_c(L_1)\neq \gder(L_1).$$

(3) To compute $\qder(L_1)$, as $f_1=f$,  we only need to replace each $a_{ij}$ with $x_{ij}$ for all
$1\leqslant i,j\leqslant 3$ and consider the equations in (\ref{eq2.5}). The generic form of an element $f\in \qder(L_1)$ is also
$$f=\begin{pmatrix}
     x_{11} & x_{12} &x_{13}   \\
     0 & x_{22} &0   \\
     x_{31} & x_{32} &x_{33}   \\
\end{pmatrix}.$$
Hence, $\qder(L_1)=\gder(L_1)$. 
\hbo}\end{exam}

\section{Embedding of Compatible Quasiderivations}\label{sec3}
\setcounter{equation}{0}
\renewcommand{\theequation}
{3.\arabic{equation}}
\setcounter{theorem}{0}
\renewcommand{\thetheorem}
{3.\arabic{theorem}}

\noindent In this section, we study the embedding question of compatible quasiderivations of 
$\upomega$-Lie algebras. A classical result appeared in \cite[Proposition 3.5]{LL00} states that
each quasiderivation of a nonassociative algebra $A$ with zero annihilator can be embedded as a derivation into a larger algebra $\wt{A}$. This section considers the enhancing question of whether 
every compatible quasiderivation of an $\upomega$-Lie algebra $L$ can be embedded as a compatible derivation into a larger $\upomega$-Lie algebra $\wt{L}$.

Let $L$ be an $n$-dimensional $\upomega$-Lie algebra over a field $\mathbb{K}$. Suppose that $k[t]$ denotes the polynomial ring in one variable $t$ over $\mathbb{K}$ and define $\wt{L}:=L\left[t\cdot k[t]/\ideal{t^3}\right]$, representing the $\mathbb{K}$-vector space of dimension $2n$ with the basis
$$\left\{x_i\cdot t^j\mid 1\leqslant i\leqslant n, 1\leqslant j\leqslant 2\right\}$$
where $x_1,\dots,x_n$ form a basis of $L$. Extending the bracket product in $L$ to $\wt{L}$ by setting
\begin{equation}
\label{ }
[x_i\cdot t, x_s\cdot t]:=[x_i,x_s]\cdot t^2,~~\textrm{ otherwise }[x_i\cdot t^j, x_s\cdot t^r]:=0
\end{equation}
and extending the bilinear form $\upomega$ on $L$ to a new bilinear form $\wt{\upomega}$ by:
\begin{equation}
\label{ }
\wt{\upomega}(x_i\cdot t, x_s\cdot t):=\upomega(x_i,x_s), ~~\textrm{ otherwise }
\wt{\upomega}\left(x_i\cdot t^j, x_s\cdot t^r\right):=0.
\end{equation}

Consider the subspace $[L,L]$ of $L$ and  the complementary space $U$ of $[L,L]$ in $L$. Then $L$ has the following decomposition of vector spaces:
\begin{equation}
\label{ }
\wt{L}=L\cdot t+L\cdot t^2=L\cdot t+(U\oplus [L,L])\cdot t^2=L\cdot t+[L,L]\cdot t^2+U\cdot t^2.
\end{equation}
We define a map $\updelta_U: \qder(L)\ra \der(\wt{L})$ that sends a quasiderivation $f$ (associated with a linear map $f'$) of $L$ to the linear map $\updelta_U(f): \wt{L}\ra\wt{L}$ given by
$$a\cdot t+b\cdot t^2+u\cdot t^2\mapsto f(a)\cdot t+f'(b)\cdot t^2$$
where $a\in L, b\in[L,L]$, and $u\in U$. 

\begin{rem}{\rm
Note that $\updelta_U(f)$ is well-defined, i.e., it is independent of the choice of $f'$. In fact, assume that
$f''$ is another linear map such that $f''([x,y])=[f(x),y]+[x,f(y)]$ for all $x,y\in L$. Then
$f''([x,y])=f'([x,y])$ for all $x,y\in L$. Hence,
$$f(a)\cdot t+f'(b)\cdot t^2=f(a)\cdot t+f''(b)\cdot t^2$$
for all $a\in L, b\in[L,L]$. Therefore, the image of an element $a\cdot t+b\cdot t^2+u\cdot t^2\in\wt{L}$ under
$\updelta_U(f)$ is unique. 
\hbo}\end{rem}

\begin{lem}\label{lem3.2}
The linear map $\updelta_U(f)$ is a derivation of $\wt{L}$, for all $f\in \qder(L)$. Moreover, $\updelta_U(f)$ is compatible if and only if $f$ is compatible. 
\end{lem}

\begin{proof}
Let us first prove that $\updelta_U(f)$ is a derivation of $\wt{L}$.
We write $(a;b,u)$ for the element $a\cdot t+b\cdot t^2+u\cdot t^2\in\wt{L}=L\cdot t+[L,L]\cdot t^2+U\cdot t^2$, where $a\in L, b\in[L,L]$, and $u\in U$. Suppose that $(a';b',u')$ denotes another element in $\wt{L}$. Then the bracket product of $\wt{L}$ shows that
$$[(a;b,u),(a';b',u')]=(0;[a,a'],0).$$ 
Thus, $\updelta_U(f)([(a;b,u),(a';b',u')])=\updelta_U(f)(0;[a,a'],0)=(0;f'([a,a']),0)$, as $f'([a,a'])\in [L,L]$. 
On the other hand, 
$[\updelta_U(f)(a;b,u),(a';b',u')]=[(f(a);f'(b),0),(a';b',u')]=(0;[f(a),a'],0)$ and 
$[(a;b,u),\updelta_U(f)(a';b',u')]=(0;[a,f(a')],0)$. Since $[f(a),a']+[a,f(a')]=f'([a,a'])$, it follows that
$\updelta_U(f)([(a;b,u),(a';b',u')])=[\updelta_U(f)(a;b,u),(a';b',u')]+[(a;b,u),\updelta_U(f)(a';b',u')]$,
which means that $\updelta_U(f)$   is a derivation of $\wt{L}$. 

Moreover, to explore the connection between compatibilities of $\updelta_U(f)$ and $f$, we observe that
\begin{eqnarray*}
&&\wt{\upomega}(\updelta_U(f)(a;b,u),(a';b',u'))+\wt{\upomega}((a;b,u),\updelta_U(f)(a';b',u'))\\
 & = & \wt{\upomega}((f(a);f'(b),0),(a';b',u')) + \wt{\upomega}((a;b,u),(f(a');f'(b'),0)) \\
 & = & \upomega(f(a),a')+\upomega(a,f(a'))
\end{eqnarray*}
for all $(a;b,u),(a';b',u')\in\wt{L}$. Therefore, $\updelta_U(f)$ is compatible if and only if $f$ is compatible.
\end{proof}

\begin{prop}\label{prop3.3}
The map $\updelta_U$ is an injective Lie homomorphism.
\end{prop}

\begin{proof}
We first show that $\updelta_U$ is linear. Suppose $c\in \mathbb K$ and $f,g\in\qder(\g)$ are arbitrary elements. Note that $\Big(\updelta_U(c\cdot f)-c\cdot \updelta_U(f)\Big)((a;b,u))=(cf(a);cf'(b),0)-c(f(a);f'(b),0)=0$ for all
$(a;b,u)\in\wt{L}$; and moreover, 
\begin{eqnarray*}
\updelta_U(f+g)((a;b,u)) & = &((f+g)(a);(f'+g')(b),0)  \\
 & = & (f(a);f'(b),0)+ (g(a);g'(b),0)\\
 &=& \updelta_U(f)(a;b,u) +\updelta_U(g)(a;b,u)
\end{eqnarray*}
where $g'$ is the linear map associated with $g$ such that $[g(x),y]+[x,g(y)]=g'([x,y])$ for all $x,y\in L$.
Hence, $\updelta_U$ is linear.

Assume that $\updelta_U(f)(a;b,u)=0$ for all $(a;b,u)\in\wt{L}$. Then 
$(f(a);f'(b),0)=0$. In particular, $f(a)=0$ for all $a\in L$, i.e., $f=0$. Thus, $\updelta_U$ is injective.

To see that $\updelta_U$ preserves Lie bracket products, we assume that $f,g\in\qder(L)$ and $(a;b,u)\in\wt{L}$. By Lemma \ref{lem2} above, we see that $[f,g]'=[f',g']$. Moreover, 
\begin{eqnarray*}
\updelta_U([f,g])(a;b,u) & = & ([f,g](a),[f,g]'(b),0) = ([f,g](a),[f',g'](b),0).
\end{eqnarray*}
On the other hands, 
\begin{eqnarray*}
[\updelta_U(f),\updelta_U(g)](a;b,u)& = & \updelta_U(f)(\updelta_U(g)(a;b,u))- \updelta_U(g)(\updelta_U(f)(a;b,u))\\
 & = & \updelta_U(f)(g(a);g'(b),0)-\updelta_U(g)(f(a);f'(b),0)\\
 &=&(fg(a);f'g'(b),0)-(gf(a);g'f'(b),0)\\
 &=& ([f,g](a),[f',g'](b),0).
\end{eqnarray*}
Hence, $\updelta_U([f,g])=[\updelta_U(f),\updelta_U(g)]$ and $\updelta_U$ is a Lie homomorphism. 
\end{proof}

Combining Lemma \ref{lem3.2} and Proposition \ref{prop3.3} obtains

\begin{coro} \label{coro3.4}
The map $\updelta_U$ restricts to a Lie subalgebra embedding of $\qder_c(L)$ into $\der_c(\wt{L})$. 
\end{coro}

The rest of this section is devoted to exploring how far the image of $\qder_c(L)$ is from $\der_c(\wt{L})$. 
Let $c(L)$ be the center of $L$, i.e., $c(L)=\{x\in L\mid [x,y]=0,\textrm{ for all }y \in L\}$.
We use $\zder(L)$ to denote the set consisting of all linear maps $f$ on $L$ such that 
$$f(L)\subseteq c(L)\textrm{ and } f([g,g])=0.$$
Clearly, $\zder(L)$ is an ideal of $\der(L)$. Furthermore,

\begin{lem}\label{lem3.5}
Each $f\in\zder(L)$ is compatible.
\end{lem}

\begin{proof}
We may assume that $\dim(L)\geqslant 3$ and $a,b,c$ are linearly independent vectors in a basis of $L$.  
If $a\in c(L)$, then the $\upomega$-Jacobi identity with $\{a,b,c\}$ shows that $\upomega(a,b)=0$. Note that
$f(L)\subseteq c(L)$, it follows that $\upomega(f(x),y)+\upomega(x,f(y))=\upomega(f(x),y)-\upomega(f(y),x)=0-0=0$, for all $x,y\in L$. Hence, $f$ is compatible.
\end{proof}

\begin{thm}\label{mainthm}
Suppose that $c(L)=0$. Then $\der_c(\wt{L})$ can be decomposed into a semidirect sum of 
the image of $\qder_c(L)$ under the map $\updelta_U$ and $\zder(\wt{L})$.
\end{thm}

\begin{proof}
The containment $\updelta_U(\qder_c(L))+\zder(\wt{L})\subseteq \der_c(\wt{L})$ follows immediately from Corollary \ref{coro3.4} and Lemma \ref{lem3.5}.  By \cite[Proposition 3.5]{LL00}, the intersection of
$\updelta_U(\qder_c(L))$ and $\zder(\wt{L})\subseteq \der_c(\wt{L})$ is zero, thus $\updelta_U(\qder_c(L))\oplus \zder(\wt{L})\subseteq \der_c(\wt{L})$.

To see the converse containment,  we note that \cite[Proposition 3.5]{LL00} asserts that any derivation $d$ of $\wt{L}$ can be written as the sum of $\updelta_U(f)$ and $h$, where $f\in \qder(L)$ and $h\in \zder(\wt{L})$. Now we assume that $d$ is compatible.  
Lemma \ref{keyl}, together with Lemma \ref{lem3.5}, implies that
$\updelta_U(f)=d-h$ is also compatible. By Lemma \ref{lem3.2}, we see that $f\in \qder_c(L)$. Hence,
$\der_c(\wt{L})\subseteq\updelta_U(\qder_c(L))\oplus \zder(\wt{L})$, and so
$$\der_c(\wt{L})=\updelta_U(\qder_c(L))\oplus \zder(\wt{L})$$
is a decomposition of vector spaces. Note that $\zder(\wt{L})$ is an ideal of $\der(\wt{L})$, thus this decomposition is actually a semidirect sum decomposition of Lie algebras. 
\end{proof}

\section{Explicit Computations in Dimension 3} \label{sec4}
\setcounter{equation}{0}
\renewcommand{\theequation}
{4.\arabic{equation}}
\setcounter{theorem}{0}
\renewcommand{\thetheorem}
{4.\arabic{theorem}}

\noindent This section provides a procedure to explicitly compute 
all generalized derivations and compatible generalized derivations of a
non-Lie 3-dimensional complex $\upomega$-Lie algebra. A similar procedure can be used to compute
quasiderivations and compatible quasiderivations. 

Our calculations are based on a classification of such $\upomega$-Lie algebras in \cite[Theorem 2]{CLZ14} in which all non-Lie 3-dimensional complex $\upomega$-Lie algebras were classified by two families ($A_\upalpha$ and $C_\upalpha$) and three exceptional  $\upomega$-Lie algebras ($L_1,L_2$, and $B$).

Consider a non-Lie finite-dimensional complex $\upomega$-Lie algebra $L$ with a basis $\{e_1,\dots,e_n\}$. Performing the following steps obtains an explicit description on $\gder(L)$:
\begin{enumerate}
  \item Compute all nonzero generating relations among these $e_i$ and determine the values of $\upomega(e_i,e_j)$ for all $i,j\in\{1,\dots,n\}$;
  \item Consider $\{(f,f_1,f_2)\mid f,f_1,f_2\in M_n(\C)\}$, where $f=(x_{ij}), f_1=(a_{ij}), f_2=(b_{ij})$, and
  $$f(e_i):=\sum_{j=1}^n x_{ji}\cdot e_j, \quad f_1(e_i):=\sum_{j=1}^n a_{ji}\cdot e_j,\quad  f_2(e_i):=\sum_{j=1}^n b_{ji}\cdot e_j.$$
   Define the ground set $V(L):=\{(e_i,e_j)\mid 1\leqslant i,j\leqslant n\}$;
  \item Verify (\ref{gder}) for all $(e_i,e_j)\in V(L)$ and use the linearly independence of $\{e_1,\dots,e_n\}$ to obtain finitely many equations  involving $x_{ij},a_{ij}$ and $b_{ij}$.  Write $\A$ for the set of all such equations;
  \item Solve the system of all equations of $\A$ only involving $x_{ij}$ and make the number of indeterminates as less as possible. 
  \item Choosing some suitable  $x_{ij}$ to eliminate other $x_{ij}$ gives us an explicit description on the generic matrix form of an element $f$ of $\gder(L)$.
\end{enumerate}

\begin{rem}\label{rem4.1}
{\rm
The first part in Example \ref{exam2.6} illustrates the above procedure for the case where $L=L_1, n=3$, $V(L)=\{(x,y),(x,z),(y,z), (y,x),(z,x),(z,y)\}$, and the equations in (\ref{eq2.5}) form the set $\A$. 
\hbo}\end{rem}

\begin{rem}\label{rem4.2}
{\rm
To calculate  $\gder_c(L)$, we need to add an extra condition (\ref{comp})  in Step 3
of the above procedure. In other words, together with the same Steps (1), (2), (4), (5), replacing Step (3) by
\begin{enumerate}
  \item[(3')] Verify (\ref{gder}) and (\ref{comp}) for all $(e_i,e_j)\in V(L)$ and use the linearly independence of $\{e_1,\dots,e_n\}$ to obtain finitely many equations  involving $x_{ij},a_{ij}$ and $b_{ij}$.  Write $\A$ for the set of all such equations;
\end{enumerate}
obtains a procedure to compute $\gder_c(L)$. Part 2 in Example \ref{exam2.6} illustrates this procedure.
\hbo}\end{rem}

We summarize our computations on $\gder(L)$ and $\gder_c(L)$ for a $3$-dimensional non-Lie complex $\upomega$-Lie algebra $L$ as in the following table: 

\begin{center}
\begin{longtable}{c|c|c|c|c}
$L$ & Elements in $\gder(L)$ & $\dim(\gder(L))$ & Elements in $\gder_c(L)$ & $\dim(\gder_c(L))$\\
\hline
$L_1$ & $\begin{pmatrix}
     x_{11} & x_{12} &x_{13}   \\
     0 & x_{22} &0  \\
     x_{31} & x_{32} &x_{33}   \\
\end{pmatrix}$ & 7 & $\begin{pmatrix}
     x_{11} & x_{12} &0   \\
     0 & -x_{11} &0   \\
     x_{31} & x_{32} &x_{33}   \\
\end{pmatrix}$ & 5 \\
\hline
$L_2$ & $\begin{pmatrix}
     x_{11} & x_{12} &x_{13}   \\
     x_{21}& x_{22} &x_{23}   \\
     0 & 0 &x_{33}   \\
\end{pmatrix}$ & 7 & $\begin{pmatrix}
     x_{11} &0 &x_{13}   \\
     x_{21} & x_{22} &x_{23}   \\
     0 & 0 &-x_{11}   \\
\end{pmatrix}$ & 5 \\
\hline
$B$ & $\begin{pmatrix}
     x_{11} & x_{12} &x_{13}   \\
     x_{21}& x_{22} &x_{23}   \\
     x_{31} & x_{32} &x_{33}   \\
\end{pmatrix}$ & 9 & $\begin{pmatrix}
     x_{11} &x_{12} &x_{13}   \\
     0 & x_{22} &x_{23}   \\
     0 & x_{32} &-x_{22}   \\
\end{pmatrix}$ & 6 \\
\hline
$A_{\upalpha}$ & $\begin{pmatrix}
     x_{11} & x_{12} &x_{13}   \\
     x_{21}& x_{22} &x_{23}   \\
     x_{31} & x_{32} &x_{33}   \\
\end{pmatrix}$ & 9 & $\begin{pmatrix}
     x_{11} &x_{12} &x_{13}   \\
     0 & x_{22} &x_{23}   \\
     0 & x_{32} &-x_{22}   \\
\end{pmatrix}$ & 6 \\
\hline
$C_{\upalpha}$ & $\begin{pmatrix}
     x_{11} & x_{12} &x_{13}   \\
     x_{21}& x_{22} &x_{23}   \\
     x_{31} & x_{32} &x_{33}   \\
\end{pmatrix}$ & 9 & $\begin{pmatrix}
     x_{11} &x_{12} &x_{13}   \\
     0 & x_{22} &x_{23}   \\
     0 & x_{32} &-x_{22}   \\
\end{pmatrix}$ & 6 \\
\captionsetup{skip=10pt}
\caption{$\gder(L)$ and $\gder_c(L)$ in Dimension 3} \label{tab1}
\end{longtable}
\end{center}

\vspace{-0.6cm}

We obtain the following consequences directly from Table \ref{tab1} above.

\begin{coro}
Let $L$ be a non-Lie $3$-dimensional complex $\upomega$-Lie algebra. Then
$\gder(L)=\gl(L)$ if and only if $L\notin\{L_1,L_2\}$.
\end{coro}

\begin{coro}
Let $L$ be a non-Lie $3$-dimensional complex $\upomega$-Lie algebra. Then
$$\gder(L)\neq\gder_c(L).$$
\end{coro}

Let's continue to consider $L$ as a non-Lie finite-dimensional complex $\upomega$-Lie algebra with a basis $\{e_1,\dots,e_n\}$. Performing the following steps may obtain an explicit description on $\qder(L)$:
\begin{enumerate}
  \item Compute all nonzero generating relations among these $e_i$ and determine the values of $\upomega(e_i,e_j)$ for all $i,j\in\{1,\dots,n\}$;
  \item Consider $\{(f,f')\mid f,f'\in M_n(\C)\}$, where $f=(x_{ij}), f'=(a_{ij})$, and
  $$f(e_i):=\sum_{j=1}^n x_{ji}\cdot e_j, \quad f'(e_i):=\sum_{j=1}^n a_{ji}\cdot e_j.$$
   Define the ground set $W(L):=\{(e_i,e_j)\mid 1\leqslant i<j\leqslant n\}$;
  \item Verify (\ref{qder}) for all $(e_i,e_j)\in W(L)$ and use the linearly independence of $\{e_1,\dots,e_n\}$ to obtain finitely many equations  involving $x_{ij},a_{ij}$ and $b_{ij}$.  Write $\B$ for the set of all such equations;
  \item Solve the system of all equations of $\B$ only involving $x_{ij}$ and make the number of indeterminates as less as possible. 
  \item Choosing some suitable  $x_{ij}$ to eliminate other $x_{ij}$ gives us an explicit description on the generic matrix form of an element $f$ of $\qder(L)$.
\end{enumerate}
Replacing Step (3) above by
\begin{enumerate}
  \item[(3'')] Verify (\ref{qder}) and (\ref{comp}) for all $(e_i,e_j)\in W(L)$ and use the linearly independence of $\{e_1,\dots,e_n\}$ to obtain finitely many equations  involving $x_{ij},a_{ij}$ and $b_{ij}$.  Write $\B$ for the set of all such equations;
\end{enumerate}
may be used to compute $\qder_c(L)$.

\begin{exam}{\rm
Let us recall the generating relations in $L_2$ appeared in \cite[Theorem 2]{CLZ14}:
$$[x,y]=0,[x,z]=y,[y,z]=z\textrm{ and }\upomega(x,z)=1, \upomega(x,y)=\upomega(y,z)=0.$$
Suppose that $$f=\begin{pmatrix}
     x_{11} & x_{12} &x_{13}   \\
     x_{21} & x_{22} &x_{23}   \\
     x_{31} & x_{32} &x_{33}   \\
\end{pmatrix} \textrm{and } f'=\begin{pmatrix}
     a_{11} & a_{12} &a_{13}   \\
     a_{21} & a_{22} &a_{23}   \\
     a_{31} & a_{32} &a_{33}   \\
\end{pmatrix}.$$
Note that $W(L_2)=\{(x,y),(x,z),(y,z)\}$, thus substituting these three elements into  (\ref{qder}) obtains 
$$
\B=\{x_{31}=x_{32}=x_{12} - a_{23}=x_{21} - a_{32}=x_{11} + x_{33} - a_{22}=x_{22} + x_{33} - a_{33}=a_{12}=a_{13}=0\}.$$
Hence, a quasiderivation $f$ of $L_2$ is of the following form
$$f=\begin{pmatrix}
     x_{11} & x_{12} &x_{13}   \\
     x_{21} & x_{22} &x_{23}   \\
     0 & 0 &x_{33}   \\
\end{pmatrix},$$
together with the associated linear map $f'$ of the following form
$$f'=\begin{pmatrix}
     a_{11} & 0 &0  \\
     a_{21} & x_{11}+x_{33} &x_{12}   \\
     a_{31} & x_{21} &x_{22}+x_{33}   \\
\end{pmatrix}.$$
Thus, $\qder(L_2)$ is a $7$-dimensional vector space. Note that the dimension of $\gder(L_2)$ is $7$ in Table \ref{tab1} and $\qder(L_2)\subseteq \gder(L_2)$. Hence, $\qder(L_2)=\gder(L_2)$.
\hbo}\end{exam}

The following Table \ref{tab2} summaries the generic forms of all  elements $f\in\qder(L)$ and their associated maps $f'$.

\begin{center}
\begin{longtable}{c|c|c}
$L$ & $f\in\qder(L)$ & Associated $f'$\\
\hline
$L_1$ & $\begin{pmatrix}
     x_{11} & x_{12} &x_{13}   \\
     0 & x_{22} &0  \\
     x_{31} & x_{32} &x_{33}   \\
\end{pmatrix}$ &  $\begin{pmatrix}
  a_{11}    & 0 & 0 \\
   a_{21}   & x_{11} + x_{22} &-x_{13} \\
   a_{31}   &-x_{31} &x_{22} + x_{33} 
\end{pmatrix}$\\
\hline
$L_2$ & $\begin{pmatrix}
     x_{11} & x_{12} &x_{13}   \\
     x_{21} & x_{22} &x_{23}   \\
     0 & 0 &x_{33}   \\
\end{pmatrix}$ &  $\begin{pmatrix}
     a_{11} & 0 &0  \\
     a_{21} & x_{11}+x_{33} &x_{12}   \\
     a_{31} & x_{21} &x_{22}+x_{33}   \\
\end{pmatrix}$\\
\hline
$B$ & $\begin{pmatrix}
     x_{11} & x_{12} &x_{13}   \\
     x_{21} & x_{22} &x_{23}   \\
     x_{31} & x_{32} &x_{33}   \\
\end{pmatrix}$  &$\begin{pmatrix}
     x_{22}+x_{33} &-x_{31} &x_{21}+x_{31}   \\
     x_{12}-x_{13} & x_{11} + x_{22} + x_{32} &x_{23}+x_{33}-x_{22}-x_{32}   \\
     x_{12} & x_{32} &x_{11}+x_{33}-x_{32}   \\
\end{pmatrix}$\\
\hline
$A_{\upalpha}$ & $\begin{pmatrix}
     x_{11} & x_{12} &x_{13}   \\
     x_{21}& x_{22} &x_{23}   \\
     x_{31} & x_{32} &x_{33}  \\
\end{pmatrix}$  & {\small $\begin{array}{c}\begin{pmatrix}
    x_{11} + x_{22} + x_{32}- \upalpha x_{31} &a_{12}&a_{13}   \\
     x_{32} & x_{11}+x_{32}+x_{33} &x_{12}+ \upalpha x_{32}  \\
     -x_{31} & x_{21}+x_{31} &\upalpha x_{31}+x_{22} + x_{33}   \\
\end{pmatrix}\\
\textrm{where }a_{12}=\upalpha x_{21}-x_{22}+x_{23}+\upalpha x_{31}-x_{32}+x_{33}\\
a_{13}= \upalpha (x_{12}-\upalpha x_{31}+x_{32}+x_{33}) + x_{12} - x_{13}
\end{array}$}\\
\hline
$C_{\upalpha}$& $\begin{pmatrix}
     x_{11} & x_{12} &x_{13}   \\
     x_{21}& x_{22} &x_{23}   \\
     x_{31} & x_{32} &x_{33}   \\
\end{pmatrix}$ & $\begin{pmatrix}
     x_{22}+x_{33} &-x_{31} &\upalpha^{-1} x_{21}   \\
     -x_{13} &x_{11}+x_{22} &\upalpha^{-1} x_{23}   \\
     \upalpha x_{12} & \upalpha x_{32} &\upalpha(x_{11}+x_{33})   \\
\end{pmatrix}$\\
\captionsetup{skip=10pt}
\caption{$\qder(L)$ in Dimension 3} \label{tab2}
\end{longtable}
\end{center}

\vspace{-0.6cm}

Combining Tables \ref{tab1} and \ref{tab2}, we obtain

\begin{coro}
Let $L$ be a non-Lie $3$-dimensional complex $\upomega$-Lie algebra. Then
$\qder(L)=\gder(L)$ and $\qder_c(L)=\gder_c(L)$.
\end{coro}

\vspace{2mm}
\noindent \textbf{Acknowledgements}. 
This research was partially supported by the University of Saskatchewan under grant No. APEF-121159. 
The third author would like to thank her Ph.D supervisor Professor Dancheng Lu for his encouragement and help.
The authors  thank the anonymous referee for their careful reading and constructive comments. 


\raggedright

\begin{bibdiv}
  \begin{biblist}
  
  
\bib{BPS24}{article}{
   author={Basdouri, Imed},
   author={Peyghan, Esmaeil},
   author={Sadraoui, Mohamed A.},
   title={Twisted Lie algebras by invertible derivations},
   journal={Asian-Eur. J. Math.},
   volume={17},
   date={2024},
   number={6},
   pages={Paper No. 2450041, 24 pp},
}
  

  \bib{CCZ21}{article}{
   author={Chang, Hongliang},
   author={Chen, Yin},
   author={Zhang, Runxuan},
   title={A generalization on derivations of Lie algebras},
   journal={Electron. Res. Arch.},
   volume={29},
   date={2021},
   number={3},
   pages={2457--2473},
}

\bib{CNY23a}{article}{
   author={Chen, Zhiqi},
   author={Ni, Junna},
   author={Yu, Jianhua},
   title={Description of $\upomega$-Lie algebras},
   journal={J. Geom. Phys.},
   volume={192},
   date={2023},
   pages={Paper No. 104926, 13 pp},
}

\bib{CNY23b}{article}{
   author={Chen, Zhiqi},
   author={Ni, Junna},
   author={Yu, Jianhua},
   title={The $\upomega$-Lie algebra defined by the commutator of an $\upomega$-left-symmetric algebra is not perfect},
   date={2023},
   pages={\texttt{arXiv:2301.12953}.},
}


\bib{CW23}{article}{
   author={Chen, Zhiqi},
   author={Wu, Yang},
   title={The classification of $\upomega$-left-symmetric algebras in low
   dimensions},
   journal={Bull. Korean Math. Soc.},
   volume={60},
   date={2023},
   number={3},
   pages={747--762},
}


  \bib{Che25}{article}{
   author={Chen, Yin},
   title={Some Lie algebra structures on symmetric powers},
   journal={Amer. Math. Monthly},
  volume={132},
   date={2025},
  number={3},
   pages={150--161}
}


\bib{CLZ14}{article}{
   author={Chen, Yin},
   author={Liu, Chang},
   author={Zhang, Runxuan},
   title={Classification of three-dimensional complex $\upomega$-Lie algebras},
   journal={Port. Math.},
   volume={71},
   date={2014},
   number={2},
   pages={97--108},
}







  \bib{CRSZ24}{article}{
   author={Chen, Yin},
   author={Ren, Shan},
   author={Shan, Jiawen},
   author={Zhang, Runxuan},
   title={Derivative maps of finite-dimensional Lie algebras},
   journal={Submitted for publication},
   date={2024},
}


\bib{CZZZ18}{article}{
   author={Chen, Yin},
   author={Zhang, Ziping},
   author={Zhang, Runxuan},
   author={Zhuang, Rushu},
   title={Derivations, automorphisms, and representations of complex
   $\upomega$-Lie algebras},
   journal={Comm. Algebra},
   volume={46},
   date={2018},
   number={2},
   pages={708--726},
}



\bib{CZ17}{article}{
   author={Chen, Yin},
   author={Zhang, Runxuan},
   title={Simple $\upomega$-Lie algebras and $4$-dimensional $\upomega$-Lie
   algebras over $\Bbb{C}$},
   journal={Bull. Malays. Math. Sci. Soc.},
   volume={40},
   date={2017},
   number={3},
   pages={1377--1390},
}



\bib{CZ23}{article}{
   author={Chen, Yin},
   author={Zhang, Runxuan},
   title={A commutative algebra approach to multiplicative Hom-Lie algebras},
   journal={Linear Multilinear Algebra},
   volume={71},
   date={2023},
   number={7},
   pages={1127--1144},
}

  \bib{CZ24a}{article}{
   author={Chen, Yin},
   author={Zhang, Runxuan},
   title={Cohomology of left-symmetric color algebras},
   date={2024},
   pages={\texttt{arXiv:2408.04033}},
}

  \bib{CZ24b}{article}{
   author={Chen, Yin},
   author={Zhang, Runxuan},
   title={Deformations of left-symmetric color algebras},
   date={2024},
   pages={\texttt{arXiv:2411.10370}},
}






\bib{LL00}{article}{
   author={Leger, George F.},
   author={Luks, Eugene M.},
   title={Generalized derivations of Lie algebras},
   journal={J. Algebra},
   volume={228},
   date={2000},
   number={1},
   pages={165--203},
}




\bib{Nur07}{article}{
   author={Nurowski, Pawel},
   title={Deforming a Lie algebra by means of a 2-form},
   journal={J. Geom. Phys.},
   volume={57},
   date={2007},
   number={5},
   pages={1325--1329},
}


\bib{Oub24}{article}{
   author={Oubba, Hassan},
   title={Local (2-Local) derivations and automorphisms and biderivations of complex $\upomega$-Lie algebras},
   journal={Matematiche (Catania)},
   volume={79},
    number={1},
   date={2024},
   pages={135--150},
}

\bib{RZ24}{article}{
 author={Ren, Shan},
   author={Zhang, Runxuan},
   title={Skew-symmetric solutions of the classical Yang-Baxter equation and $\mathcal{O}$-operators of Malcev algebras},
   journal={Filomat},
   volume={38},
   date={2024},
   number={14},
   pages={5003--5019},
}

\bib{Vla22}{article}{
   author={Vladeva, Dimitrinka},
   title={Endomorphisms of upper triangular matrix semirings},
   journal={Comm. Algebra},
   volume={50},
   date={2022},
   number={2},
   pages={822--835},
}


\bib{Vla24}{article}{
   author={Vladeva, Dimitrinka},
   title={Endomorphisms of upper triangular matrix rings},
   journal={Beitr. Algebra Geom.},
   volume={65},
   date={2024},
   number={2},
   pages={291--306},
}



\bib{Zha21}{article}{
   author={Zhang, Runxuan},
   title={Representations of $\upomega$-Lie algebras and tailed derivations of
   Lie algebras},
   journal={Internat. J. Algebra Comput.},
   volume={31},
   date={2021},
   number={2},
   pages={325--339},
}

\bib{ZZ10}{article}{
   author={Zhang, Runxuan},
   author={Zhang, Yongzheng},
   title={Generalized derivations of Lie superalgebras},
   journal={Comm. Algebra},
   volume={38},
   date={2010},
   number={10},
   pages={3737--3751},
}



\bib{ZC21}{article}{
   author={Zhou, Jia},
   author={Chen, Liangyun},
   title={On low-dimensional complex $\upomega$-Lie superalgebras},
   journal={Adv. Appl. Clifford Algebr.},
   volume={31},
   date={2021},
   number={3},
   pages={Paper No. 54, 30 pp},
}

\bib{ZCM18}{article}{
   author={Zhou, Jia},
   author={Chen, Liangyun},
   author={Ma, Yao},
   title={Generalized derivations of Hom-Lie triple systems},
   journal={Bull. Malays. Math. Sci. Soc.},
   volume={41},
   date={2018},
   number={2},
   pages={637--656},
}

\bib{ZCMS18}{article}{
   author={Zhou, Jia},
   author={Chen, Liangyun},
   author={Ma, Yao},
   author={Sun, Bing},
   title={On $\upomega$-Lie superalgebras},
   journal={J. Algebra Appl.},
   volume={17},
   date={2018},
   number={11},
   pages={1850212, 17 pp},
}





  \end{biblist}
\end{bibdiv}
\end{document}